\documentclass[12pt]{amsart}
\usepackage{amstext,amsfonts,amssymb,amscd,amsbsy,amsmath,verbatim, mathrsfs, fullpage}
\usepackage[alphabetic,abbrev,lite,backrefs]{amsrefs} 
\usepackage{ifthen,tikz}
\usepackage{color}
\usepackage{amsthm}
\usepackage{latexsym}
\usepackage[all]{xy}
\usepackage{enumerate}
\usepackage{mathtools}

\numberwithin{equation}{section}

\newtheorem{lemma}{Lemma}[section]

\newtheorem{prop}[lemma]{Proposition}
\newtheorem{cor}[lemma]{Corollary}

\newtheorem{claim*}{Claim}
\newtheorem{thm}[lemma]{Theorem}

\newtheorem{question}[lemma]{Question}

\theoremstyle{definition}
\newtheorem{defn}[lemma]{Definition}
\newtheorem{example}[lemma]{Example}
\newtheorem{notation}[lemma]{Notation}

\theoremstyle{remark}
\newtheorem{remark}[lemma]{Remark}

\newcommand{\Tate}{{\mathbf{T}}}

\newcommand{\cC}{\mathcal{C}}

\newcommand{\m}{\mathfrak m}
\newcommand{\PP}{\mathbb P}

\newcommand{\bA}{\mathbb A}
\newcommand{\A}{\bA}

\newcommand{\ZZ}{\mathbb Z}
\newcommand{\QQ}{\mathbb Q}

\newcommand{\Proj}{\operatorname{Proj}}

\newcommand{\Tor}{\operatorname{Tor}}

\newcommand{\Hom}{\operatorname{Hom}} 

\newcommand{\cO}{{\mathcal O}}

\newcommand{\depth}{\operatorname{depth}}
\newcommand{\F}{\FF}

\newcommand{\defi}[1]{\textsf{#1}} 

\newcommand{\beq}{\begin{displaymath}}
\newcommand{\eeq}{\end{displaymath}}

\def\nc{\newcommand}
\def\on{\operatorname}
\nc{\Q}{\mathbb{Q}}
\nc{\RR}{\mathbf{R}}
\nc{\LL}{\mathbf{L}}
\nc{\xra}{\xrightarrow}
\nc{\xla}{\xleftarrow}

\def\om{\omega}

\def\DM{\operatorname{DM}}

\def\th{\on{th}}
\def\F{\mathcal{F}}

\def\l{\ell}

\nc{\into}{\hookrightarrow}
\nc{\onto}{\twoheadrightarrow}
\nc{\OO}{\mathcal{O}}
\nc{\Z}{\mathbb{Z}}
\nc{\cA}{\mathcal{A}}
\nc{\w}{\widehat}
\nc{\End}{\on{End}}
\nc{\res}{\frac{1}{x_0x_1}}
\nc{\tF}{\widetilde{F}}
\nc{\tG}{\widetilde{G}}
\nc{\tf}{\widetilde{f}}
\nc{\Com}{\on{Com}}

\nc{\G}{\mathbb{G}}
\nc{\cG}{\mathcal{G}}
\nc{\cE}{\mathcal E}
\nc{\cF}{\mathcal F}
\nc{\cR}{\mathcal R}
\nc{\cD}{\mathcal D}
\nc{\cB}{\mathcal B}
\nc{\cT}{\mathcal T}
\nc{\cL}{\mathcal L}

\nc{\bM}{\mathbf M}
\nc{\bN}{\mathbf N}
\nc{\U}{\mathbf U}
\nc{\BM}{\mathbf B \mathbf M}
\nc{\Dsg}{\on{D}_{\on{sg}}}
\nc{\fC}{\mathcal{C}}
\nc{\fG}{\mathcal{G}}
\nc{\N}{\mathbb{N}}

\nc{\del}{\partial}
\nc{\cone}{\on{cone}}
\nc{\D}{\on{D}_{\on{diff}}}
\nc{\DMb}{\on{D}^b_{\DM}}
\nc{\Db}{\on{D}^{\on{b}}}
\nc{\Kb}{\on{K}^{\on{b}}}
\nc{\fm}{\mathfrak{m}}
\nc{\Flag}{\on{Flag}}
\nc{\DMmin}{\DM_{\on{min}}}
\nc{\Ddiff}{\on{D}_{\on{diff}}}
\nc{\Dbdiff}{\on{D}^\on{b}_{\on{diff}}}
\nc{\wO}{\widehat{\OO}}
\nc{\wT}{\widehat{T}}
\nc{\from}{\leftarrow}
\nc{\wLL}{\widetilde{\LL}}
\nc{\augCech}{\widetilde{\cC}}
\nc{\Fold}{\on{Fold}}
\nc{\Ext}{\on{Ext}}
\nc{\FF}{\mathbf{F}}
\nc{\Comper}{\Com_{\on{per}}}
\nc{\Unfold}{\on{Unfold}}
\nc{\intHom}{\underline{\Hom}}
\nc{\Ex}{\on{Ex}}
\nc{\tg}{\widetilde{g}}

\nc{\B}{\mathcal{B}}
\nc{\K}{\mathcal{K}}
\nc{\kos}{\on{Kos}}
\nc{\Perf}{\on{Perf}}
\nc{\tR}{\widetilde{\cR}}
\nc{\X}{\mathcal{X}}
\nc{\Cl}{\on{Cl}}
\nc{\fU}{\mathcal{U}}
\nc{\bU}{\mathbf U}

\nc{\st}{\on{st}}

\nc{\coh}{\on{coh}}

\def\D{\mathcal{D}}

\nc{\tU}{\U}
\nc{\bC}{\mathbf{C}}
\nc{\aux}{\on{aux}}

\title{Positivity and nonstandard graded Betti numbers}
\thanks{The first author was supported by NSF-RTG grant 1502553.  The second author was supported by NSF grants 
DMS-1601619 and DMS-1902123.}
\author{Michael K. Brown}
\address{Department of Mathematics, Auburn University, Auburn, AL, 36849}
\email{mkb0096@auburn.edu}

\author{Daniel Erman}
\address{Department of Mathematics, University of Wisconsin, Madison, WI, 53706}
\email{derman@math.wisc.edu}

\makeatletter
\@namedef{subjclassname@2020}{%
  \textup{2020} Mathematics Subject Classification}
\makeatother

\begin{document}

\maketitle

\begin{abstract}
A foundational principle in the study of modules over standard graded polynomial rings is that 
geometric positivity conditions imply vanishing of Betti numbers.  The main goal of this paper is to determine the extent to which this principle extends to the nonstandard $\Z$-graded case. In this setting, the classical arguments break down, and the results become much more nuanced. We 
introduce a new notion of Castelnuovo-Mumford regularity
 and 
 employ exterior algebra techniques to control the shapes of nonstandard $\Z$-graded minimal free resolutions. Our main result reveals a unique feature in the nonstandard $\Z$-graded case: the possible degrees of the syzygies of a graded module in this setting are controlled not only by its regularity, but also by its depth. As an application of our main result, we show that, given a simplicial projective toric variety and a module $M$ over its coordinate ring, the multigraded Betti numbers of $M$ are contained in a particular polytope when $M$ satisfies an appropriate positivity condition.
\end{abstract}

\section{Introduction}
The goal of this paper is to clarify some aspects of the relationship between regularity and syzygies in the case of a nonstandard $\ZZ$-grading.  We begin with two overarching questions:

\begin{question}\label{q:intro 1}   Consider a closed subvariety $X$ of a weighted projective space.  How does knowledge about vanishing of the sheaf cohomology of $X$ translate into bounds on the degrees of the defining equations of $X$?
\end{question}
\begin{question}\label{q:intro 2}
Consider a module $M$ over a $\ZZ$-graded polynomial ring.  What can we say about the Betti numbers of high degree truncations $M_{\geq r}(r)$ for $r\gg 0$?
\end{question}
In the standard graded case, both questions may be answered via the theory of Castelnuovo-Mumford regularity. For Question~\ref{q:intro 1}: if $H^i(\PP^n, \cO_X(r-i))=0$ for all $i>0$, then $X$ can be defined by equations of degree $\leq r+1$. For Question~\ref{q:intro 2}:  if $r$ is at least the regularity of $M$, then $M_{\geq r}$ has a linear free resolution.  For details, one can see~\cite{EG,laz,mumford} and more.  Yet neither question has a satisfying answer in the nonstandard $\ZZ$-graded case.

Benson introduced an analogue of Castelnuovo-Mumford regularity in the nonstandard $\ZZ$-graded case~\cite{benson}, which we will refer to as \defi{weighted regularity} to emphasize the distinction with the standard graded theory.\footnote{This is also a special case of the notion of multigraded regularity defined by Maclagan-Smith \cite{MS}.} While that notion has had tremendous applications in certain areas (see, for example, Symonds' work \cite{symonds}), it does not provide sharp answers to either of the above questions.  In short, there are some natural features of regularity in the standard graded case that are lacking in the weighted case.  
See, for instance, Remark~\ref{rmk:not sharp}.

We propose an alternate analogue of Castelnuovo-Mumford regularity---\defi{Koszul regularity}---that provides sharper answers to the above questions.  We do not suggest that this notion should supersede weighted regularity. In fact, a main theme from recent work on syzygies with nonstandard gradings, e.g.~\cite{BC, linear, Np, BES, BHS, BS, EES, HNV, MS, SV}, is that notions from the standard graded case can have several distinct nonstandard graded analogues, each of which is useful for different purposes.  For instance, several analogues of linear resolutions in the nonstandard $\Z$-graded setting play a key role in~\cite{Np}.  Our goal in this paper is to demonstrate how an alternate analogue of regularity---Koszul regularity---can provide sharper information in some contexts.

Let us set up our notation more precisely.  Given integers  $1 \le d_0 \le \cdots \le d_n$, we let $\mathbb P(\mathbf{d})=\mathbb P(d_0, \dots, d_n)$ denote the associated weighted projective space over a field $k$. Let $S=k[x_0, \dots, x_n]$ denote its Cox ring, where $\deg(x_i)=d_i$, and $\mathfrak m$ the homogeneous maximal ideal of $S$.  The following definition was introduced by Benson~\cite[\S5]{benson}:

\begin{defn}\label{defn:weighted reg}
We say that $M$ is \defi{weighted $r$-regular} if $H^i_{\m}(M)_j = 0$ for $i \ge 0$ and $j > r - i$. The \defi{weighted Castelnuovo-Mumford regularity of $M$} is the smallest $r$ such that $M$ is $r$-regular.
\end{defn}
Weighted regularity has had significant applications, for instance to group cohomology and invariant theory~\cite{benson, symonds,symonds2} as well as to our work on $N_p$-conditions in weighted projective spaces~\cite{Np}.   To define Koszul regularity, we need the following notation:

\begin{notation}
\label{wi}
Let $w^i$ (resp. $w_i$) be the sum of the $i$ largest (resp. smallest) degrees of the variables: that is, $w^i\coloneqq\sum_{j=n-i+1}^n d_j$ and $w_i\coloneqq\sum_{j=0}^{i-1} d_j$. By convention, $w_0 = 0 = w^0$, and  $w_{-1} =-1 = w^{-1}$.  If $K$ is the Koszul complex resolving the residue field of $S$, then $w^i$ (resp. $w_i$) is the maximal (resp. minimal) degree of a generator of $K_i$.
\end{notation}

\begin{defn}\label{defn:0 reg}
Let $M$ be a graded $S$-module. We say $M$ is \defi{Koszul $r$-regular} if $H^i_\fm (M)_d=0$ for all $d\geq r -w^{i-1}$. The \defi{Koszul regularity} of $M$ is the minimal $r$ such that $M$ is Koszul $r$-regular.
 \end{defn}

In the standard graded case, where each $d_i$ is 1, Definitions~\ref{defn:weighted reg} and~\ref{defn:0 reg} specialize to the standard definition of regularity.\footnote{See also \cite{MS19} for yet another distinct notion of regularity in the weighted case.}
In general, Koszul $r$-regularity is a stronger condition than weighted $r$-regularity. See Example~\ref{ex:comparison} for a comparison of these notions in a simple case. 

Our main result is the following, which uses the theory of Koszul regularity to convert cohomological vanishing conditions into vanishing results on Betti numbers.

\begin{thm}
\label{introweighted}
Let $M$ be a finitely generated, graded $S$-module with Betti numbers $\beta_{i,j}(M) = \Tor_i^S(M, k)_j$.   If $M$ is Koszul $r$-regular, then 
\[
\beta_{i,j}(M) = 0 \text{ for } j \ge r + w^{i+\depth(M)} - w^{\depth(M)-1}.
\]
\end{thm}
While weighted regularity provides a bound on the number of rows of the Betti table \cite[Proposition 1.2]{symonds2}, Theorem~\ref{introweighted} is often sharper if one is interested in bounds on specific Betti numbers: see, for instance, Corollaries~\ref{cor:defining equations} and \ref{cor:truncation}.

One unusual feature of Theorem \ref{introweighted} is its implication that the degrees of the syzygies of a module are governed not only by its regularity, but also by its \emph{depth}. This is invisible in the standard graded context, as $w^{i +\depth(M)} - w^{\depth(M)-1}$ is always $i + 1$ in that case, irrespective of $\depth(M)$.  Another consequence of Theorem \ref{introweighted} is that the bounds on Betti numbers shadow the degrees arising in the Koszul complex of the variables. This is in contrast with~\cite[Proposition 1.2]{symonds2}, which uses regularity to give a bound on the number of rows of the Betti table, but not a distinct bound for each homological degree.

When $S$ is standard graded, Theorem \ref{introweighted} precisely recovers one direction of the well-known equivalence between the local cohomology and free resolution definitions of regularity as in~\cite[Theorem 1.2(1)]{EG}.
But, for a general $\mathbf{d}$, the converse of Theorem \ref{introweighted} is simply false; in fact, one cannot determine Koszul regularity solely from the Betti table (see Example~\ref{ex:same betti}).  However, a partial converse does hold, where the integers $w^i$ are replaced by the integers $w_i$; see Theorem~\ref{technical weighted}(b).  The gap between the $w^i$ and the $w_i$ measures the degree to which the equivalence between the local cohomology and free resolution definitions of regularity gets distorted in the nonstandard $\ZZ$-graded case.

The classical argument of Eisenbud-Goto~\cite[Theorem 1.2(1)]{EG} that proves Theorem~\ref{introweighted} in the standard graded case simply does not extend to the weighted setting. The basic problem is that there are fewer homogeneous linear forms in the weighted case; see Remark~\ref{proofremark} for details. Our proof of Theorem \ref{introweighted} is therefore totally distinct from that of~\cite{EG}; we use exterior algebra methods, applying the Tate resolution technology developed in~\cite{tate}. Curiously, this flips a script from Eisenbud-Fl\o ystad-Schreyer's work~\cite{EFS}: 
 we use Tate resolutions on weighted spaces $\mathbb P(\mathbf d)$ to understand resolutions of truncations, whereas \cite{EFS} uses properties of truncations from~\cite{EG} to \emph{define} Tate resolutions on $\PP^n$.

Finally, Theorem \ref{introweighted} enables us to provide sharper answers to our initial Questions \ref{q:intro 1} and \ref{q:intro 2}.  For Question~\ref{q:intro 1}, we have:

\begin{cor}\label{cor:defining equations}
Let $X$ be a closed subvariety of the weighted projective space $\Proj(S)$ with defining ideal $I_X\subseteq S$.  If $S/I_X$ has Koszul regularity $r$, then $I_X$ is generated in degrees $<r+w^2$.
\end{cor}
Indeed, we obtain Corollary~\ref{cor:defining equations} by applying Theorem~\ref{introweighted} with $i = 1$ and observing that $w^{1+\depth(M)} - w^{\depth(M)-1} \le w^2$. As for Question~\ref{q:intro 2}, we prove the following:
\begin{cor}
\label{cor:truncation}
Let $M$ be a finitely generated, graded $S$-module. For any $r\gg 0$, the module $M_{\geq r}(r)$ is Koszul $0$-regular, and thus $\beta_{i,j}(M_{\geq r}(r))\ne 0$ only if $w_i \leq j <w^{i+1}$.
\end{cor}
In the classical setting, Corollary~\ref{cor:truncation} implies that $M_{\geq r}(r)$ has a linear resolution, as $w_i=i=w^i$ for all $i$.  Thus, the conditions in this corollary can be seen as providing a nonstandard graded analogue of a linear resolution; in fact this notion of a ``Koszul linear'' complex arises in~\cite{Np} in relation to $N_p$-conditions on weighted projective space, and it contrasts with the notion of strong linearity from~\cite[Definition 1.2]{linear}.

\begin{example}\label{ex:shift}
A natural question arising from Corollary~\ref{cor:truncation} is: where does the homological shift come from? That is: why is the upper bound for Betti numbers in homological degree $i$ given by $<w^{i+1}$, as opposed to $\leq w^i$?  The need for this shift can be seen via a simple example.  Let $S=k[x,y]$, where $\deg(x)=1$ and $\deg(y)=10$.  The degrees of the generators of the truncation $S_{\geq r}(r)$ depend on the remainder of $r$ divided by $10$.  For instance, $S_{\geq 1}(1)$ is generated by $x$ and $y$ in degrees $0$ and $9$, whereas $S_{\geq 7}(7)$ is generated by $x^7$ and $y$ in degree $0$ and $3$.  Thus, the generating degrees of $S_{\geq r}(r)$---i.e. the Betti numbers in homological degree 0---depend on the maximal degree of a variable, i.e. $w^1$.  The shift in Theorem~\ref{introweighted} is even more dramatic, depending as it does on the depth of the module.
\end{example}

\medskip 

Benson's definition of weighted regularity was a source of inspiration for Maclagan-Smith's work on multigraded regularity~\cite{MS}, as well as many followup results, e.g.~\cite{symonds, symonds2}.  It would be interesting to consider whether an analogue of Koszul regularity in the multigraded setting might also yield new results like Corollaries~\ref{cor:defining equations} and \ref{cor:truncation}.  In \S\ref{subsec:multigraded betti}, we pursue a related line of inquiry.  Specifically, we show how Theorem~\ref{introweighted} can be applied to the study of Betti numbers over the Cox rings of more general toric varieties, resulting in Theorem~\ref{introthm:truncated res general ALT}.

\subsection*{Acknowledgements} We thank Christine Berkesch, Juliette Bruce, David Eisenbud, Lauren Cranton Heller, Mahrud Sayrafi, Gregory G.~Smith, and Frank Olaf-Schreyer for valuable conversations.

\section{Background}
\label{background}

\subsection{Regularity and related notions}\label{subsec:regularity}
In this subsection, we provide background on the various flavors of regularity that appear in this paper, and we discuss some examples that clarify the distinctions between them.  First, it will be useful to recall the definition of regularity in the standard graded case:
\begin{defn} [The standard graded case]
\label{regdef}
Assume that $\deg(x_i) = 1$ for all $i$. 
Given $r \in \Z$, the following conditions on a finitely generated, graded $S$-module $M$ are equivalent:
\begin{enumerate}
\item $H^i_{\m}(M)_j = 0$ for $i \ge 0$ and $j > r - i$,
\item $\Tor_i^S(M, k)_j = 0$ for $i \ge 0$ and $j > r + i$. 
\end{enumerate}
We say $M$ is \defi{$r$-regular} if it satisfies these equivalent conditions. The \defi{Castelnuovo-Mumford regularity of $M$} is the smallest $r$ such that $M$ is $r$-regular.
\end{defn}

Let us now return to the nonstandard $\ZZ$-graded case and consider some results and examples to clarify the definitions of weighted regularity and Koszul regularity from the introduction.  For weighted regularity, Symonds proved the following:

\begin{prop}[\cite{symonds2} Proposition 1.2]\label{prop:symonds}
Let $\sigma= \sum_{i=0}^n (d_i-1)$, and let $M$ be a finitely generated, graded $S$-module. The module $M$ is weighted $r$-regular if and only if $\Tor_i(M,k)_j=0$ for all $j > r+i+\sigma$.
\end{prop}
Thus, weighted regularity measures the number of rows of the Betti table of $M$.

\begin{remark}\label{rmk:not sharp}
In the standard graded case, if $M$ is finite length and $r$-regular, then it is possible that $\Tor_i^S(M,k)_{r+i}\ne 0$ for any $0 \leq i\leq n+1$.  In other words, the bounds on $\Tor$ can be sharp in every degree.  However, for certain choices of $\mathbf{d}$, this can fail in the nonstandard graded case.  To take a simple example, let $S=k[x,y]$ with $\deg(x)=\deg(y)=3$, and consider a finite length module $M$ of weighted regularity $0$.  Consider its minimal free resolution:
$
F_0\gets F_1\gets F_2\gets 0.
$
Proposition~\ref{prop:symonds} implies that $F_1$ is generated in degrees $<1+\sigma = 5$ and that $F_2 $ is generated in degrees $<2+\sigma = 6$.  But since $F_2$ is nonzero and the degrees of the variables are $3$, we see that the highest allowable degree of a generator of $F_1$ is actually $6-3=3$.  In other words, for certain classes of modules, and intermediate homological degrees, the bounds from Proposition~\ref{prop:symonds} might always fail to be sharp.
\end{remark}

\begin{example} \label{ex:same betti}
Let $S=k[x_0,x_1]$ with degrees $1$ and $2$, and let $M=S(-1)/(x_{1}) \oplus S(-2)$.  Both $\mathfrak m$ and $M$ have the same Betti table:
\[
\beta(\mathfrak m) =
\footnotesize\begin{matrix}
         & 0 & 1   \\
      1: & 1 & .  \\
      2: & 1& 1 \\
      \end{matrix}
      = \beta(M)
      \normalsize
\]
Thus, they have the same weighted regularity; because $\sigma=1$ in this case, the weighted regularity is $1$. However, $\mathfrak m$ is Koszul $1$-regular, while $M$ is only Koszul $2$-regular; one readily sees this by applying Theorem~\ref{technical weighted} below, or by a direct calculation using Local Duality.  This shows that Koszul regularity cannot be detected solely from the Betti table, in general; however, Theorem~\ref{technical weighted} below implies that Koszul regularity \emph{can} be detected from the Betti table provided that the module is Cohen-Macaulay.
\end{example}

\begin{example}\label{ex:comparison}
We provide a quick comparison of weighted and Koszul regularity via a local cohomology computation.  Let $S=k[x_0,x_1,x_2,x_3]$ and $M=S/(x_2,x_3)$.  We note that $H^i_{\mathfrak m}(M)_j \ne 0$ if and only if $i=2$ and $j\leq -d_0 -d_1$.\footnote{The generator of $H^2_{\mathfrak m}(M)$ may be viewed as the monomial $\frac{1}{x_0x_1}$.}  Using this, we see that $M$ has weighted regularity $2-d_0-d_1$.   To compute Koszul regularity, we first note that $w^1=d_3$.  So the Koszul regularity of $M$ is the minimal $r$ such that $H^2_{\mathfrak m}(M)_j = 0$ for all $j\geq r - d_3$; that is, $r=d_3-d_0-d_1+1$.
\end{example}

\begin{remark}
One feature of Koszul regularity is that it is homogeneous in the following sense: if we rescale the degrees of the variables of $S$ by $\deg(x_i) \mapsto \lambda \deg(x_i)$, and we rescale the grading of an $S$-module $M$ by $\lambda$ as well, then the Koszul regularity of $M$ is also rescaled by $\lambda$.  This is not true for weighted regularity.
\end{remark}

\subsection{The multigraded BGG correspondence}
\label{BGG}
Let $E$ denote the $\Z \oplus \Z$-graded exterior algebra  $\Lambda_k(e_0, \dots, e_n)$ with $\deg(e_i) = (-\deg(x_i); -1)$. We let $\Com(S)$ denote the category of complexes of $\Z$-graded $S$-modules and $\DM(E)$ the category of \defi{differential $E$-modules}, i.e. $E$-modules $D$ equipped with a degree $(0; -1)$ endomorphism $\del$ such that $\del^2 = 0$. Given an object $D \in \DM(E)$, we let $H(D)$ denote its homology.

As proven by \cite{HHW}, there is a multigraded analogue of the Bernstein-Gel'fand-Gel'fand (BGG) correspondence that gives an adjunction
$$
\LL : \DM(E) \leftrightarrows \Com(S) : \RR.
$$
We refer the reader to \cite[\S 2]{tate} for a detailed introduction to the multigraded BGG correspondence. We will not be concerned with the functor $\LL$ in this paper, and we will only need the formula for $\RR(M)$ when $M$ is an $S$-module, which is given as follows. Let $\om_E$ denote the $E$-module $\underline{\Hom}_k(E, k)\cong E(-\sum_{i = 0}^n \deg(x_i); -n-1)$. The object $\RR(M) \in \DM(E)$ has underlying $E$-module $\bigoplus_{a \in \Z} M_a \otimes_k \om_E(-a; 0)$ and differential given by $m \otimes f \mapsto \sum_{i = 0}^n x_im\otimes e_if$.

A key point is that $\Z$-graded Betti numbers may be computed via BGG:

\begin{prop}[\cite{tate} Proposition 2.11(a)]
\label{tor} 
Let $M$ be an $S$-module. We have an identification $H(\RR(M))_{(a; j)} = \on{Tor}_j^S(M, k)_a$ of $\Z \oplus \Z$-graded $k$-vector spaces.
\end{prop}

\subsection{Tate resolutions on weighted projective stacks}
\label{Tateback}

The BGG functor $\RR$ admits a geometric refinement: the \defi{Tate resolution functor} $\Tate\colon \coh(X) \to \DM(E)$. Tate resolutions over toric varieties/stacks are introduced in \cite[\S 3]{tate}, and we refer the reader there for a full introduction to the topic, and to~\cite{ABI,BE} for additional background on differential modules. Here, we briefly discuss Tate resolutions over weighted projective stacks.  The following result summarizes the key features of Tate resolutions we will need:

\begin{thm}[\cite{tate} Theorems 3.3 and 3.7]
\label{tatethm}
Let $\F$ be a coherent sheaf on the weighted projective stack $X = \mathcal{P}(d_0, \dots, d_n)$,  i.e. the stack quotient of $\A^{n+1} \setminus \{0\}$ by the action of the multiplicative group $k \setminus \{0\}$ given by $\lambda \cdot (x_0, \dots, x_n) = (\lambda^{d_0}x_0, \dots, \lambda^{d_n}x_n)$.
\begin{enumerate}
\item The Tate resolution $\Tate(\F)$ is an exact, minimal differential $E$-module such that $H^i(X, \F(j)) = \underline{\Hom}(k, \Tate(\F))_{(j; -i)}$.
\item Choose an $\m$-saturated $S$-module $M$ such that $\widetilde{M} =  \F$. The Tate resolution $\Tate(\F)$ is isomorphic to the mapping cone of a minimal free resolution $F \xra{\simeq} \RR(M)$ of the differential $E$-module $\RR(M)$.
\end{enumerate}
\end{thm}

See \cite[Appendix B]{tate} for background
 on differential $E$-modules and
 \cite[Examples 3.11 - 3.13]{tate} for examples of Tate resolutions over weighted projective stacks. 

\begin{remark}
\label{quasicoh}
The coherence assumption on $\F$ in Theorem \ref{tatethm} can be loosened. Indeed, the general construction of Tate resolutions on projective toric stacks in \cite[\S 3.2]{tate} makes sense even for quasi-coherent sheaves, and the proof of \cite[Theorem 3.3]{tate} works verbatim at this level of generality, so Theorem \ref{tatethm}(1) holds for any quasi-coherent sheaf. Additionally, if $\F$ is a quasi-coherent sheaf on $X$ satisfying
\begin{enumerate}
\item $\F = \widetilde{M}$ for some $S$-module $M$ with $H^0_\fm(M) = 0$, and 
\item there exists $N \gg 0$ such that $H^i_\fm(M)_d = 0$ for all $i > 0$ and $d \ge N$;
\end{enumerate}
then the proof of \cite[Theorem 3.7]{tate} works essentially verbatim as well, and so Theorem~\ref{tatethm}(2) also holds in this more general setting. We use this in the proof of Theorem~\ref{technical weighted}.
\end{remark}

\section{Proof of Theorem \ref{introweighted}}
\label{sec:regweighted}
We will prove the following strengthened version of Theorem~\ref{introweighted}:

\begin{thm}
\label{technical weighted}
Let $k$ be a field, and let $S = k[x_0, \dots, x_n]$, $\Z$-graded so that $d_i \coloneqq \deg(x_i) \ge 1$ for all $i$. Let $M$ be a graded $S$-module. \begin{enumerate}
\item[(a)] If $M$ is $r$-Koszul regular, and $H^0_\fm(M)_j = 0$ for $j \ll 0$, then $\Tor_i^S(M, k)_j = 0$ for $j \ge r + w^{i+\depth(M)} - w^{\depth(M)-1}$.
\item[(b)] Suppose $M$ is finitely generated. If $M$ is Cohen-Macaulay, then the converse of (a) holds. In general, if $\Tor_i^S(M, k)_j = 0$ for $j \ge r + w^{i + 1}$ (so, for instance,  if $\Tor_i^S(M, k)_j = 0$ for $j \ge r + w^{i+\depth(M)} - w^{\depth(M)-1}$), then $H^i_\fm (M)_d=0$ for all $d\geq r -w_{i-1}$.
\end{enumerate}
\end{thm}

Let us briefly sketch the ideas that led us to this result.  Let $M$ be Koszul $0$-regular and generated in degree $0$.  Recall that the Tate resolution of the sheaf associated to $M$ is an exact, bigraded differential module over an exterior algebra: under certain conditions, it is the cone of a free resolution of the form
$
G \overset{\epsilon}{\longrightarrow} \RR(M)
$ (see Theorem~\ref{tatethm} and Remark~\ref{quasicoh}).  The Koszul $0$-regularity of $M$ constrains the degrees of the generators of $G$.  This also constrains the degrees of the image of $\epsilon$; since the homology of $\RR(M)$ encodes $\Tor^S_*(M,k)_*$, this in turn bounds the Betti numbers of $M$.  The appearance of the integers $w^i$ and $w_i$ in the Theorem arise from working over the exterior algebra. Our actual proof is based on this basic idea, but it requires some rather technical bookkeeping.

\begin{proof}[Proof of Theorem \ref{technical weighted}]
Twisting $M$ appropriately, we may assume $r = 0$. Let us prove (a). 
By the Horseshoe Lemma applied to 
$
0\to H^0_\fm( M) \to M \to M/H^0_\fm (M)\to 0,
$
it suffices to prove the statement for $H^0_\fm (M)$ and in the case where $H^0_\fm (M)=0$.  Our regularity assumption implies that $H^0_\fm (M)$ has a maximal degree $d$ such that $H^0_\fm (M)_d \ne 0$; by our convention $w^{-1}=-1$, we have $d \le 0$. We have a short exact sequence
$
0\to H^0_\fm (M)_d \to H^0_\fm (M) \to N \to 0.
$
Since $k(-d)$ is resolved by the Koszul complex twisted by $-d$, the statement holds for the minimal free resolution of $H^0_\fm (M)_d$, because $\beta_{i,j}(k) =0$ for $j>w^i$. We now apply the same argument to $N$; since  $H^0_\fm (M)$ has a minimal degree where it is nonzero, this process eventually terminates.
We may therefore assume that $H^0_\fm (M) =0$. 

By Lemma~\ref{tor}, the Betti numbers of $M$ are encoded by the homology of $\RR (M)$; it thus suffices to prove that 
\[
H ( \RR (M))_{(a;j)} \ne 0 \text{ only if } a < w^{j+\depth(M)} - w^{\depth(M)-1}.
\]

By Theorem \ref{tatethm} (and Remark \ref{quasicoh})\footnote{Our regularity assumption on $M$ implies that the condition in Remark \ref{quasicoh}(2) is satisfied.}, the Tate resolution of $\widetilde{M}$ is isomorphic to the mapping cone of a minimal free resolution $\epsilon\colon G \xra{\simeq} \RR (M)$, and the generators of $G$ are in bijection with sheaf cohomology groups of $\widetilde{M}$. Observe that $\omega_E(-a; j+1)$ is a summand of $G$ only if $H^{j+1}_\fm (M)_a\ne 0$. By the regularity assumption on $M$, we have $a<-w^j$ in this case. The generator of $\omega_E$ has degree $(w;n+1)$, and so the generator of $\omega_E(-a;j+1)$ has degree $(w+a; n-j)$.  Applying the inequality $a<-w^j$, we get:
\[
\l \coloneqq w+a < w - w^j = w_{n+1-j}.
\]
Setting $i=n-j$, we arrive at the following key point: every generator $\tau$ of $G$ of degree $(\l;i)$ satisfies $\l<w_{i+1}$.   We remark, for use in a moment, that $j\geq \depth(M)-1$. 

Every class in $H(\RR (M))$ may be represented by an element in the image of $\epsilon$; in particular, we can write every element in $H (\RR (M))$ as a sum of elements of the form $f \cdot \epsilon(\tau)$, where $\tau$ is a generator of $G$, and $f\in E$. Say $\deg(\tau) = (\l ; i)$ and $\deg(f) = (-m; -t)$. Since $\RR (M)$ has no elements of degree $(u;v)$ with $v < 0$, the same is true for $H (\RR (M))$. 
We therefore have  $-i \le -t \le 0$. 
Since $f\tau$ has degree $(\ell - m; i-t)$, our goal is to show that $\l-m < w^{i-t+\depth(M)} - w^{\depth(M)-1}$. The maximum possible value for $-m$ is $-w_t$.   Since $\tau$ is a generator of degree $(\l;i)$, the argument in the previous paragraph implies that $\l <w_{i+1}$.  We now compute:
\[
\l-m  <w_{i+1} - w_t= \sum_{c=0}^{i} d_c -\sum_{c=0}^{t-1} d_c=\sum_{c = t}^{i} d_c.
\]
Since $j\geq \depth(M)-1$, we have $i=n-j \leq n-\depth(M) +1$. Moreover, we have: 
\begin{align*}
\sum_{c = t}^{i} d_c \leq \sum_{c = t+1}^{i+1} d_c \leq \cdots &\leq \sum_{c= {n-(i-t)-\depth(M)+1}
}^{n-\depth(M)+1} d_c\\
&=\sum_{c=n-(i-t)-\depth(M)+1}^n d_c- \sum_{c=n-\depth(M)+2}^n d_c \\
&=w^{i-t+\depth(M)} - w^{\depth(M)-1}.
\end{align*}
Thus, $\l-m < w^{i-t+\depth(M)} - w^{\depth(M)-1}$, which is what we wanted to show.

As for (b): by Grothendieck vanishing, we may assume $i \ge \depth(M)$. By Local Duality, we have
$H^i_\fm(M) = \Ext^{n+1-i}_S(M, S(-w))^*$, where $w \coloneqq w^{n+1}$. Thus, $H^i_\fm(M)_d$ is a subquotient of $\bigoplus_{j \in \Z} (S(j - w)^*)_d^{ \beta_{n + 1 - i, j}}= \bigoplus_{j \in \Z} (S_{j - w - d})^{ \beta_{n + 1 - i, j}}.$
By hypothesis, the $j^{\th}$ summand vanishes unless $j < w^{n + 1 - i + \depth(M)}- w^{\depth(M) - 1}$, and so we assume this inequality holds.

Now, assume $M$ is Cohen-Macaulay. Again by Grothendieck vanishing, we may assume $i = \depth(M)$, in which case $j < w - w^{\depth(M) - 1}$. Thus, when $d \ge -w^{\depth(M)-1}$, we have $j -w - d< 0$, and so $H^{\depth(M)}_\fm(M)_d = 0$. It follows that $M$ is 0-Koszul regular. 

In general, when $d \ge -w_{i -1}$, we have
$$
j - w - d  < w^{n + 1 - i + \depth(M)}- w^{\depth(M) - 1} - w + w_{i - 1}\le w^{n + 2 -i }  - w + w_{i - 1} = 0.
$$
Thus, $H^i_\fm(M)_d = 0$.
\end{proof}

\begin{remark}
\label{proofremark}
The proof of Theorem \ref{introweighted}(b) is virtually identical to that of the ``only if" direction of Eisenbud-Goto's Theorem \cite[Theorem 1.2(1)]{EG}. However, we emphasize that the proof of the ``if" direction of Eisenbud-Goto's Theorem does \emph{not} generalize to the weighted setting, and so our approach to proving Theorem \ref{introweighted}(a) is radically different from that of \cite{EG}. Indeed, the proof of the ``if" direction of \cite[Theorem 1.2(1)]{EG} makes crucial use of the fact that, if $M$ is a finitely generated module over a standard graded polynomial ring with positive depth, and the ground field is infinite, then there exists a homogeneous linear form $\l$ such that $\l$ acts as a non-zero-divisor on $M$.
This is false in our context: for example, say $X = \PP(2,3,5)$, and let $M = S / I$, where $I = (x_0,x_1) \cap (x_0,x_2) \cap (x_1,x_2)$. The non-zero-divisors of $S/I$ are those elements not in $(x_0,x_1) \cup (x_0,x_2) \cup (x_1,x_2)$. For instance, $f = x_0^{15} + x_1^{10} + x_2^{6}$ is such an element, and in fact there is no homogeneous non-zero-divisor on $M$ of smaller degree than $f$. 
\end{remark}

\begin{example}\label{ex:235}
Let $S=k[x_0,x_1,x_2]$ with degrees $2,3$ and $5$.  One can check that $\mathfrak m$ is Koszul $1$-regular and has depth $1$.  Theorem~\ref{introweighted}(a) thus implies that the maximal degree of a generator of the $i^{\th}$ syzygies of $\mathfrak m$ is $w^{i+1}$; for $i=0,1,2$ this yields bounds of $5,8$ and $10$, respectively.  Each of these bounds is sharp, as the minimal free resolution of $\mathfrak m$ has the form
\[
S(-2)\oplus S(-3)\oplus S(-5) \gets S(-5)\oplus S(-7)\oplus S(-8) \gets S(-10) \gets 0.
\]
\end{example}

Let us now prove Corollaries~\ref{cor:defining equations} and~\ref{cor:truncation}.

\begin{proof}[Proof of Corollary~\ref{cor:defining equations}]
Theorem~\ref{introweighted} implies that $\beta_{1,j}(S/I_X) = 0$ for
\[
j \ge r+w^{1+\depth(M)} - w^{\depth(M)-1}. 
\]
Since $S/I_X$ has depth at least $1$, and $w^2 \ge w^{i+1}-w^{i-1}=d_{i}+d_{i-1} $ for all $i \ge 1$, 
we conclude that $\beta_{1,j}(S/I_X) = 0$ for $j \ge r + w^2$, i.e. $I_X$ is generated in degrees $<r+w^2$.
\end{proof}
Our proof of Corollary~\ref{cor:truncation} requires the following lemma.  
\begin{lemma}
\label{hightrunc}
Let $M$ be a finitely generated, graded $S$-module. The truncation $M_{\ge r}$ is Koszul $r$-regular for $r \gg 0$. 
\end{lemma}

\begin{proof}
We will show $M_{\ge r}(r)$ is Koszul 0-regular for $r \gg 0$. For $i > 1$, we have $H^i_{\fm}(M_{\ge r}(r))_d = H^i(X, \widetilde{M}(d + r))$. It therefore follows from (the weighted version of) Serre Vanishing that we can choose $r \gg 0$ so that, for any $i > 1$, we have $H^i_{\fm}(M_{\ge r}(r))_d = 0$ for $d \ge -w^{i -1}$. Since $H^0_{\fm}(M)$ has finite length, we may also choose $r \gg 0$ such that $H^0_\fm(M_{\geq r}(r))=0$, in which case $H^1_\fm(M_{\geq r}(r))$ is supported entirely in negative degrees.
\end{proof}

\begin{proof}[Proof of Corollary~\ref{cor:truncation}]
Let $K$ denote the Koszul complex on the variables $x_0, \dots, x_n$. Since the minimal degree of a generator of $M$ is $0$, the minimal degree of an element of $\Tor_i^S(M, k) = H_i(M \otimes_S K)$ is $w_i$; this yields the lower bound on $a$.  The upper bound follows immediately from Theorem~\ref{introweighted}(a) and Lemma~\ref{hightrunc}.
\end{proof}

\section{An application to Betti numbers over Cox rings of toric varieties}
\label{subsec:multigraded betti} 

Let $X$ be a simplicial, projective toric variety. In this section, we let $S = k[x_0, \dots, x_n]$ denote the $\Cl(X)$-graded Cox ring of $X$ and $B \subseteq S$ the irrelevant ideal of $X$. Our next goal is to prove a version of Theorem~\ref{introweighted} for $\Cl(X)$-graded $S$-modules; the idea is to use the theory of primitive collections to reduce to the $\Z$-graded case. 

We recall that a \defi{primitive collection} for $X$ may be described algebraically as a subset of the variables $x_0, \dots, x_n$ generating an associated prime of $B$. Primitive collections were first studied by Batyrev in~\cite{batyrev}; we refer the reader to \cite[Definition 5.1.5]{CLS} for background. Our assumption that $X$ is simplicial and projective ensures that we may apply the theory of primitive collections in our setting. If $I \subseteq \{x_0, \dots, x_n\}$ is a primitive collection, then we can use \cite[p. 305]{CLS} to define a homomorphism $\deg_I\colon \Cl(X)\to \ZZ$.  More specifically, each primitive collection $I$ induces, via \cite[Definition~6.4.10 and (6.4.8)]{CLS}, a coefficient vector $(b'_0, b'_1, \dots, b'_n)\in \QQ^{n+1}$.  By minimally clearing denominators, we obtain $(b_0,b_1 \dots, b_n)\in \ZZ^{n+1}$, and we define $\deg_I(x_i)=b_i$; this gives a well-defined map $\Cl(X)\to \ZZ$ by the exactness of the sequence in~\cite[(6.4.1)]{CLS}. We have $\deg_I(x_j)>0$ for $x_j \in I$ and $\deg_I(x_j) \leq 0$ for $x_j \notin I$. In particular, given a primitive collection $I$, the map $\deg_I$ makes $S_I = S / \langle x_i \notin I \rangle$ into a positively $\ZZ$-graded ring.

\begin{example}\label{ex:prim collections}
Let $X$ be the Hirzebruch surface of type $3$. The Cox ring of $X$ is $S = k[x_0, \dots, x_3]$, with $\Z^{ 2}$-grading given by $\deg(x_0) = (1,0) = \deg(x_2)$, $\deg(x_1) = (-3, 1)$, and $\deg(x_3) = (0,1)$. The irrelevant ideal of $X$ is $(x_0, x_2) \cap (x_1, x_3)$. There are therefore two primitive collections for $X$: $\{x_0,x_2\}$ and $\{x_1,x_3\}$. The map $\deg_{\{x_0,x_2\}}$ (resp. $\deg_{\{x_1,x_3\}}$) is projection onto the first (resp. second) coordinate.
\end{example}

We introduce the following notation:

\begin{notation}
\label{Ikoszul}
Let $I$ be a primitive collection. We set
$$
w^{j}_I = \begin{cases}
\max\{ \sum_{x_i\in I'} \deg_I(x_i) \text{ : } I'\subseteq I \text{ and } \#I' = j\}, & j < \#I; \\
\sum_{x_i\in I} \deg_I(x_i), & j\geq \#I.
\end{cases}
$$  
\end{notation}

The following toric analogue of Theorem~\ref{introweighted} is the main result of this section:

\begin{thm}\label{introthm:truncated res general ALT}
Let $X$ be a simplicial, projective toric variety with Cox ring $S$ and irrelevant ideal $B$. Let $M$ be a $\Cl(X)$-graded $S$-module, and assume $H_B^0(M) = 0$. Fix a primitive collection $I$, and let $P_I$ denote the corresponding minimal prime of $B$.   We have the following\footnote{
Just as in Theorem \ref{introweighted}(a), the bounds on Betti numbers appearing in Theorem \ref{introthm:truncated res general ALT} may be tightened by considering the depth of $M_I$ for each primitive collection $I$. 
}:
\[
\begin{matrix}
\text{If } H^i_{P_I}(M)_a = 0 \\
 \text{ for $i > 0$ and all degrees}\\
 a\in \Cl(X) \text{ where }  \deg_I(a) \geq -w^{i - 1}_I,
\end{matrix}
\qquad 
\Rightarrow
\qquad 
\begin{matrix}
\text{then } \beta_{i,a}(M) =0 \\
 \text{ for $i\geq 0$ and all degrees}\\
 a\in \Cl(X) \text{ where }  \deg_I(a) \geq w^{i+1}_I.
\end{matrix}
\]
\end{thm}

Put more simply, Theorem \ref{introthm:truncated res general ALT} says that, if $M$ is an $S$-module satisfying appropriate positivity conditions, then the multigraded Betti numbers of $M$ must lie within a particular polytope.  Similar ideas have been appeared in~\cite{SVTW} and elsewhere.  For instance, results from~\cite{BES,BHS,EES} give analogues of linear resolutions for truncations and \cite{BC,chardin-holanda} give bounds on Betti numbers, at least in the case where $X$ is a product of projective spaces.  A fairly general result in this direction is~\cite[Theorem 1.5(2)]{MS}, but this result addresses the structure of a (potentially infinite) virtual resolution and does not yield specific results about Betti numbers.

To prove Theorem~\ref{introthm:truncated res general ALT}, we will need the following technical lemma. Given an $S$-module $M$ and a primitive collection $I$, let $M_I$ denote the module $M$ considered as an $S_I$-module.

\begin{lemma}
\label{technical}
Let $M$ be a graded $S$-module, $I \subseteq \{0, \dots, n\}$ a primitive collection, and $J$ the complement of $I$. The $\Cl(X) \oplus \Z$-graded $k$-vector space $\Tor_*^S(M, k)$ is a subquotient of $\Tor_*^{S_I}(M_I, k) \otimes_k \om_{E_J}$.
\end{lemma}

\begin{proof}
Let $K$ denote the Koszul complex on all the variables in $S$, $K_I$ the Koszul complex on $\{x_i \text{ : } i \in I\}$, and $K_J$ the Koszul complex on $\{x_i \text{ : } i \in J\}$. Think of the tensor product $(M\otimes_S K_I)\otimes_S K_J$ as a bicomplex whose totalization is $M\otimes_S K$. We have a spectral sequence
$
E^1 = \Tor^S_*(M, S_J) \otimes_S K_J \Rightarrow \Tor^S_*(M, k).
$
Notice that $\Tor^S_*(M, S_J) \otimes_S K_J = \Tor^S_*(M, S_J) \otimes_k \om_{E_J}$. Finally, it follows from the change of rings spectral sequence for $\Tor$ associated to the inclusion $S_I \into S$ that there is an isomorphism $\Tor_*^S(M, S_J) \cong\Tor_*^{S_I}(M_I, k)$.
\end{proof}

\begin{proof}[Proof of Theorem \ref{introthm:truncated res general ALT}]
Consider $S_I$ and $M_I$ as $\Z$-graded via $\deg_I : \Cl(X) \to \Z$. Theorem~\ref{technical weighted}(a) implies that $\on{\Tor}_{i}^{S_I}(M_I,k)_{j}\ne 0$ only if $j<w_I^{i+1}$. Since $\om_{E_J}$ is non-positively graded in the $\deg_I$-grading, it follows that 
$$
(\Tor_*^{S_I}(M_I, k) \otimes_k \om_{E_J})_{(a ; j)} = \bigoplus_{b \in \Z} \bigoplus_{ \l = 0}^j  \Tor_\l^{S_I}(M_I, k)_{a - b}\otimes_k (\om_{E_J})_{(b;j - \l)}
$$
is nonzero only if $a \le a - b < w_I^{j+1}$. Now apply Lemma \ref{technical}.
\end{proof}

\begin{example}\label{ex:trun res Hirz}
Let $X$ be a Hirzebruch surface of type $3$. As discussed in Example \ref{ex:prim collections}, there are two primitive collections on $X$, and the corresponding maps $\deg_I : \Cl(X) \to \Z$ correspond to projection onto the first/second coordinate. In both cases, $w^{j}_I\leq 2$ for all $j$. Let $M$ be a $\Cl(X)$-graded $S$-module that satsifies the hypotheses of Theorem \ref{introthm:truncated res general ALT} with respect to both primitive collections $I$. Theorem \ref{introthm:truncated res general ALT} implies that the generators of a minimal free resolution of $M$ lie in the following degrees:
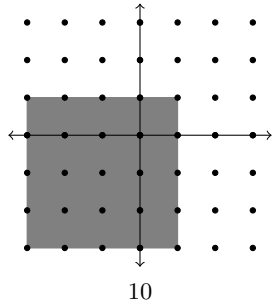
\[
\begin{tikzpicture}[scale = .5]
    \draw[fill, gray=0.1] (1,1)--(1,-3)--(-3,-3)--(-3,1)--(1,1);
    \draw[fill] (0,0) circle [radius=2pt];
    \draw[fill] (0,1) circle [radius=2pt];
    \draw[fill] (0,2) circle [radius=2pt];     
    \draw[fill] (0,3) circle [radius=2pt];
    \draw[fill] (1,0) circle [radius=2pt];
    \draw[fill] (1,1) circle [radius=2pt];
    \draw[fill] (1,2) circle [radius=2pt];     
    \draw[fill] (1,3) circle [radius=2pt];
    \draw[fill] (2,0) circle [radius=2pt];
    \draw[fill] (2,1) circle [radius=2pt];
    \draw[fill] (2,2) circle [radius=2pt];     
    \draw[fill] (2,3) circle [radius=2pt];
    \draw[fill] (3,0) circle [radius=2pt];
    \draw[fill] (3,1) circle [radius=2pt];
    \draw[fill] (3,2) circle [radius=2pt];     
    \draw[fill] (3,3) circle [radius=2pt];
    \draw[fill] (-0,0) circle [radius=2pt];
    \draw[fill] (-0,1) circle [radius=2pt];
    \draw[fill] (-0,2) circle [radius=2pt];     
    \draw[fill] (-0,3) circle [radius=2pt];
    \draw[fill] (-1,0) circle [radius=2pt];
    \draw[fill] (-1,1) circle [radius=2pt];
    \draw[fill] (-1,2) circle [radius=2pt];     
    \draw[fill] (-1,3) circle [radius=2pt];
    \draw[fill] (-2,0) circle [radius=2pt];
    \draw[fill] (-2,1) circle [radius=2pt];
    \draw[fill] (-2,2) circle [radius=2pt];     
    \draw[fill] (-2,3) circle [radius=2pt];
    \draw[fill] (-3,0) circle [radius=2pt];
    \draw[fill] (-3,1) circle [radius=2pt];
    \draw[fill] (-3,2) circle [radius=2pt];     
    \draw[fill] (-3,3) circle [radius=2pt];
    \draw[fill] (0,-0) circle [radius=2pt];
    \draw[fill] (0,-1) circle [radius=2pt];
    \draw[fill] (0,-2) circle [radius=2pt];     
    \draw[fill] (0,-3) circle [radius=2pt];
    \draw[fill] (1,-0) circle [radius=2pt];
    \draw[fill] (1,-1) circle [radius=2pt];
    \draw[fill] (1,-2) circle [radius=2pt];     
    \draw[fill] (1,-3) circle [radius=2pt];
    \draw[fill] (2,-0) circle [radius=2pt];
    \draw[fill] (2,-1) circle [radius=2pt];
    \draw[fill] (2,-2) circle [radius=2pt];     
    \draw[fill] (2,-3) circle [radius=2pt];
    \draw[fill] (3,0) circle [radius=2pt];
    \draw[fill] (3,-1) circle [radius=2pt];
    \draw[fill] (3,-2) circle [radius=2pt];     
    \draw[fill] (3,-3) circle [radius=2pt];
    \draw[fill] (-0,-0) circle [radius=2pt];
    \draw[fill] (-0,-1) circle [radius=2pt];
    \draw[fill] (-0,-2) circle [radius=2pt];     
    \draw[fill] (-0,-3) circle [radius=2pt];
    \draw[fill] (-1,-0) circle [radius=2pt];
    \draw[fill] (-1,-1) circle [radius=2pt];
    \draw[fill] (-1,-2) circle [radius=2pt];     
    \draw[fill] (-1,-3) circle [radius=2pt];
    \draw[fill] (-2,-0) circle [radius=2pt];
    \draw[fill] (-2,-1) circle [radius=2pt];
    \draw[fill] (-2,-2) circle [radius=2pt];     
    \draw[fill] (-2,-3) circle [radius=2pt];
    \draw[fill] (-3,0) circle [radius=2pt];
    \draw[fill] (-3,-1) circle [radius=2pt];
    \draw[fill] (-3,-2) circle [radius=2pt];     
    \draw[fill] (-3,-3) circle [radius=2pt];
    \draw[<->] (-3.5,0)--(3.5,0);
    \draw[<->] (0,-3.5)--(0,3.5);
\end{tikzpicture}
\]
(one should imagine this box extending infinitely down and to the left). If $M$ is also generated in degrees $\geq 0$, then the Betti numbers of $M$ must lie in the polytope
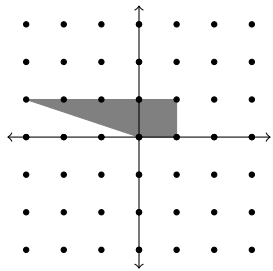
\[
\begin{tikzpicture}[scale = .5]
    \draw[fill, gray=0.1] (1,1)--(1,0)--(0,0)--(-3,1)--(1,1);
    \draw[fill] (0,0) circle [radius=2pt];
    \draw[fill] (0,1) circle [radius=2pt];
    \draw[fill] (0,2) circle [radius=2pt];     
    \draw[fill] (0,3) circle [radius=2pt];
    \draw[fill] (1,0) circle [radius=2pt];
    \draw[fill] (1,1) circle [radius=2pt];
    \draw[fill] (1,2) circle [radius=2pt];     
    \draw[fill] (1,3) circle [radius=2pt];
    \draw[fill] (2,0) circle [radius=2pt];
    \draw[fill] (2,1) circle [radius=2pt];
    \draw[fill] (2,2) circle [radius=2pt];     
    \draw[fill] (2,3) circle [radius=2pt];
    \draw[fill] (3,0) circle [radius=2pt];
    \draw[fill] (3,1) circle [radius=2pt];
    \draw[fill] (3,2) circle [radius=2pt];     
    \draw[fill] (3,3) circle [radius=2pt];
    \draw[fill] (-0,0) circle [radius=2pt];
    \draw[fill] (-0,1) circle [radius=2pt];
    \draw[fill] (-0,2) circle [radius=2pt];     
    \draw[fill] (-0,3) circle [radius=2pt];
    \draw[fill] (-1,0) circle [radius=2pt];
    \draw[fill] (-1,1) circle [radius=2pt];
    \draw[fill] (-1,2) circle [radius=2pt];     
    \draw[fill] (-1,3) circle [radius=2pt];
    \draw[fill] (-2,0) circle [radius=2pt];
    \draw[fill] (-2,1) circle [radius=2pt];
    \draw[fill] (-2,2) circle [radius=2pt];     
    \draw[fill] (-2,3) circle [radius=2pt];
    \draw[fill] (-3,0) circle [radius=2pt];
    \draw[fill] (-3,1) circle [radius=2pt];
    \draw[fill] (-3,2) circle [radius=2pt];     
    \draw[fill] (-3,3) circle [radius=2pt];
    \draw[fill] (0,-0) circle [radius=2pt];
    \draw[fill] (0,-1) circle [radius=2pt];
    \draw[fill] (0,-2) circle [radius=2pt];     
    \draw[fill] (0,-3) circle [radius=2pt];
    \draw[fill] (1,-0) circle [radius=2pt];
    \draw[fill] (1,-1) circle [radius=2pt];
    \draw[fill] (1,-2) circle [radius=2pt];     
    \draw[fill] (1,-3) circle [radius=2pt];
    \draw[fill] (2,-0) circle [radius=2pt];
    \draw[fill] (2,-1) circle [radius=2pt];
    \draw[fill] (2,-2) circle [radius=2pt];     
    \draw[fill] (2,-3) circle [radius=2pt];
    \draw[fill] (3,0) circle [radius=2pt];
    \draw[fill] (3,-1) circle [radius=2pt];
    \draw[fill] (3,-2) circle [radius=2pt];     
    \draw[fill] (3,-3) circle [radius=2pt];
    \draw[fill] (-0,-0) circle [radius=2pt];
    \draw[fill] (-0,-1) circle [radius=2pt];
    \draw[fill] (-0,-2) circle [radius=2pt];     
    \draw[fill] (-0,-3) circle [radius=2pt];
    \draw[fill] (-1,-0) circle [radius=2pt];
    \draw[fill] (-1,-1) circle [radius=2pt];
    \draw[fill] (-1,-2) circle [radius=2pt];     
    \draw[fill] (-1,-3) circle [radius=2pt];
    \draw[fill] (-2,-0) circle [radius=2pt];
    \draw[fill] (-2,-1) circle [radius=2pt];
    \draw[fill] (-2,-2) circle [radius=2pt];     
    \draw[fill] (-2,-3) circle [radius=2pt];
    \draw[fill] (-3,0) circle [radius=2pt];
    \draw[fill] (-3,-1) circle [radius=2pt];
    \draw[fill] (-3,-2) circle [radius=2pt];     
    \draw[fill] (-3,-3) circle [radius=2pt];
    \draw[<->] (-3.5,0)--(3.5,0);
    \draw[<->] (0,-3.5)--(0,3.5);
\end{tikzpicture}
\]
For instance, set $N=S/(x_0x_1)$, and let $M$ be the truncated twist $M\coloneqq N_{\geq (2,3)}(2,3)$.  A direct computation shows that $H^0_B(M) = 0$, and, for either primitive collection $I$ on $X$, $M_I$ is Koszul $0$-regular with respect to the associated $\Z$-grading. The minimal free resolution of $M$ has the form
$
S^{ 6} \gets S(3,-1)^{ 2} \oplus S(0,-1)^{ 3} \oplus S(-1,0)^{ 5} \gets S(2,-1) \oplus S(-1,-1)^{ 3} \gets 0.
$
\end{example}
\bibliographystyle{amsalpha}
\bibliography{Bibliography}

\end{document}